\date{September 11, 2010}

\documentclass[a4paper,11pt,reqno]{amsart}
\usepackage{latexsym,amsmath,amsfonts,amscd,amssymb}
\usepackage{graphics}
\usepackage[all]{xy}
\usepackage{booktabs}

\newcommand\R{\mathbb{R}}

\newcommand{\la}{\langle}
\newcommand{\ra}{\rangle}

\newcommand{\bd}{\partial}
\newcommand{\nn}{\textbf{n}}
\newcommand{\bfr}{\mathbf{r}}

\newtheorem{teo}{Theorem}[section]
\newtheorem{lem}[teo]{Lemma}
\newtheorem{prop}[teo]{Proposition}
\newtheorem{cor}[teo]{Corollary}
\newtheorem{rem}[teo]{Remark}
\theoremstyle{definition}

\newtheorem{defi}[teo]{Definition}

\author[M. Calvo]{Mar\'{\i}a Calvo}
\address{Departamento de \'Algebra, Facultad de Ciencias,
Universidad de Granada, 18071 Granada, Spain}
\email{mariacc88@gmail.com}

\author[V. Mu\~{n}oz]{Vicente Mu\~{n}oz}
\address{Facultad de
Matem\'aticas, Universidad Complutense de Madrid, Plaza de Ciencias
3, 28040 Madrid, Spain}
\email{vicente.munoz@mat.ucm.es}


\thanks{First author partially supported through Spanish MEC grant MTM2007-63582}

\title[The most inaccessible point]{The most inaccessible point of a convex domain} 
\begin{document}

\maketitle

\begin{abstract}
The inaccessibility of a point $p$ in a bounded domain $D\subset \R^n$
is the minimum of the lengths of segments through $p$ with boundary at $\bd D$. The
points of maximum inaccessibility $I_D$ are those
where the inaccessibility achieves its maximum.
We prove that for strictly convex domains, $I_D$ is either a point or a segment,
and that for a planar polygon $I_D$ is in general a point.
We study the case of a triangle, showing that this point is not any
of the classical notable points.
\end{abstract}

\section{Introduction} \label{sec:intro}

The story of this paper starts when the second author was staring at some workers
spreading cement over the floor of a square place to construct a
new floor over the existing one. The procedure was the following: first they divided the
area into triangular areas (actually quite irregular triangles, of around 50 square meters of
area). They put bricks all along the sides of the triangles and then poured the
liquid cement in the interior. To make the floor flat, they took a big rod of metal,
and putting it over the bricks on two of the sides, they moved the rod to flatten the
cement. Of course, they had to be careful as they were reaching the most inner part
of the triangle.

The question that arose in this situation is: what is the minimum size for the rod?
Even more, which is the most inaccessible point, i.e. the one that requires the full
length of the rod? Is it a notable point of the triangle?

\medskip

The purpose of this paper is to introduce the concept of maximum inaccessibility for a
domain. This is done in full generality for a bounded domain in $\R^n$.
The inaccessibility function $\bfr$ assigns to a point of the domain $D$ the minimum
length of a segment through it with boundary in $\bd D$. We introduce the sets $D_r=\{x \, | \, \bfr(x)>r\}$
and the most inaccessible set $I_D$ given by the points where the
inaccessibility function achieves its maximum value (the notion has to be suitable modified
for the case where $\bfr$ only has suppremum).

Then we restrict to convex domains to prove convexity properties of the sets $D_r$
and $I_D$. For strictly convex domains, $I_D$ is either
a point or a segment. For planar convex domains not containing pairs of regular points with parallel
tangent lines (e.g. polygons without parallel sides), $I_D$ is
a point. In some sense, domains for which $I_D$ is not a point are of very special nature.
When $I_D=\{p_D\}$ is a point, we call $p_D$ \emph{the point of maximum inaccessibility} of $D$.


In the final section, we shall study in detail the case of a polygonal domain in the plane,
and more specifically the case of a triangle, going back to the original problem.
One of the results is that the point $p_T$, for a triangle $T$, is not a notable point of $T$.
It would be nice to determine explicitly this point in terms of the coordinates of the vertices.
We do it in the case of an isosceles triangle.

\noindent \textbf{Acknowledgements}
We are very grateful to Francisco Presas for many discussions which helped to shape and
improve the paper.
The first author thanks Consejo Superior de Investigaciones
Cient\'{\i}ficas for its hospitality and for funding (through the JAE-Intro program)
her stay in July and September 2009 at the Instituto de Ciencias Matem\'aticas CSIC-UAM-UC3M-UCM,
where part of this work was done. Special thanks to Francisco Presas for supervision during her visit.

\section{Accessibility for domains} \label{sec:access}

Let $D\subset \R^n$ be a bounded domain, that is an open subset such that $\overline{D}$ is compact.
Clearly also $\bd D$ is compact. For a point $p\in D$, we consider the function:
  $$
   f_p: S^{n-1} \to \R_+\, ,
  $$
which assigns to every unit vector $v$ the length $l(\gamma)$
of the segment $\gamma$ given as the connected
component of $(p+\R v) \cap D$ containing $p$.

\begin{lem}\label{lem:semi-cont-1}
  $f_p$ is lower-semicontinuous, hence it achieves its minimum.
\end{lem}

\begin{proof}
Let us introduce some notation: for $p\in D$ and $v\in S^{n-1}$, we denote
$\gamma_{p,v}$ the connected
component of $(p+\R v) \cap D$ containing $p$. (So $\overline{\gamma_{p,v}}=[P,Q]$
for some $P,Q\in \bd D$.)
Now define the function
 $$
 H: D\times S^{n-1} \to \R_+\, ,
 $$
by $H(p,v)=f_p(v)$. Let us see that $H$ is lower-semicontinuous.
Suppose that $(p_n,v_n) \to (p,v)$. Let $\overline{\gamma_{p_n,v_n}}=[P_n,Q_n]$,
where $P_n,Q_n\in \bd D$. As $\bd D$ is compact, then there are convergent
subsequences (which we denote as the original sequence), $P_n\to P$, $Q_n\to Q$.
Clearly $P,Q\in \bd D$. Let $\gamma$ be the open segment with $\overline\gamma=[P,Q]$.
Then $p\in \gamma \subset (p+\R v)$. So $\gamma_{p,v}\subset \gamma$ and
  $$
  H(p_n,v_n)=l(\gamma_{p_n,v_n})=|| P_n-Q_n||  \to || P-Q|| = l(\gamma) \geq l(\gamma_{p,v})=H(p,v)\, .
  $$
Clearly, $f_p(v)=H(p,v)$, obtained by freezing $p$, is also lower-semicontinuous.
\end{proof}

\begin{rem}
 In Lemma \ref{lem:semi-cont-1}, if $D$ is moreover convex, then $H$ is continuous. This follows
 from the observation that a closed segment $\sigma=[P,Q]$
 with endpoints $P,Q\in \bd D$ either is fully contained in $\bd D$
 or $\sigma \cap \bd D =\{P,Q\}$. The segment $\gamma$ in the proof of Lemma \ref{lem:semi-cont-1}
 has endpoints in $\bd D$ and goes through $p$, therefore it coincides with $\gamma_{p,v}$.
 So $H(p_n,v_n)\to H(p,v)$, proving the continuity of $H$.
\end{rem}

We say that a point $p\in D$ is \emph{$r$-accessible} if there is a segment of length at most $r$
with boundary at $\bd D$ and containing $p$. Equivalently, let
  $$
  \bfr (p)=\min_{v\in S^{n-1}} f_p(v)\, ,
  $$
which is called \emph{accessibility} of $p$. Then $p$ is $r$-accessible if $\bfr(p)\leq r$.
Extend $\bfr$ to $\overline{D}$ by setting $\bfr(p)=0$ for $p\in \bd D$.

\begin{prop}\label{prop:semi-cont-2}
 The function $\bfr:\overline{D}\to \R_{\geq 0}$ is lower-semicontinuous.
\end{prop}

\begin{proof}
 We first study the function $\bfr: D\to \R_+$. As $\bfr(p)=\min_v H(p,v)$, the lower-semicontinuity of
 $H$ gives the lower-semicontinuity of $\bfr$ : If $p_n\to p$, take $v_n$ such that $\bfr(p_n)=H(p_n,v_n)$.
 After taking a subsequence, we can assume that $(p_n,v_n)\to (p,v)$. So
  $$
  \underline{\lim}\ \bfr(p_n) = \underline{\lim}\ H(p_n,v_n)\geq H(p,v) \geq \bfr(p)\, ,
  $$
 as required.

 Finally, as we define $\bfr(p)=0$ if $p\in\bd D$, those points give no problem to lower-semicontinuity.
\end{proof}

We have some easy examples where $f_p$ or $\bfr$ are not continuous. For instance, if we consider the domain
 $$
 D= \{(x,y) | x^2+y^2< 1, x \leq 0\}\cup \{(x,y) | x^2+y^2< 4, x> 0\} \, ,
 $$
and let $p=(0,0)$. Then $f_p:S^1 \to \R^+$ has constant value $3$ except at the horizontal vectors where it
has value $2$. Also $\bfr$ is not continuous, since $\bfr(p)=2$, but $\bfr((\epsilon, 0)) \approx 3$, for
$\epsilon>0$ small.

\begin{rem} \label{rem:a}
 If $D$ is convex, then $\bfr:D\to \R_+$ is continuous. Let $p_n\to p$. Take $w$ so that $H(p,w)=\bfr(p)$. Then
 $\bfr(p)=H(p,w)=\lim H(p_n,w) \geq \overline{\lim}\ \bfr(p_n)$, using the continuity of $H$ and $H(p_n,w)\geq
 \bfr(p_n)$. So $\bfr$ is upper-semicontinuous, and hence continuous.
\end{rem}

The function $\bfr:\overline{D} \to \R_{\geq 0}$ may not be continuous, even for convex domains. Take
a semicircle $\{(x,y) | x^2+y^2< 1, x> 0\}$. Then $\bfr((\epsilon,0))=1$, for $x>0$ small, but $\bfr((0,0))=0$.

\medskip

We introduce the sets:
\begin{eqnarray*}
  D_r &=& \{ p\in {D} \, | \, \bfr(p)> r\}, \\
  E_r &=& \overline{\{ p\in D \, | \, \bfr(p)\geq r\}}.
\end{eqnarray*}
$D_r$ is open by Proposition \ref{prop:semi-cont-2}, and $E_r$ is compact.
The function $\bfr$ is clearly bounded, so it has a suppremum.

\begin{defi}
 We call $R=\sup \bfr$ the \emph{inaccessibility} of $D$. We call
 $$
 I_D:= \bigcap_{r<R} E_r
 $$
the set of points of \emph{maximum inaccessibility} of $D$.
\end{defi}

\bigskip

The set $I_D$ may intersect the boundary of $D$. For instance, $D=\{(x,y) |
x^2+y^2 < 1, x> 0\}$. Then $R=1$. It can be seen that $I_D=E_R=\{(x,0)| 0 \leq x\leq \frac{\sqrt{3}}2\}$.

Moreover, $I_D$ can be a point of the boundary. Take $D=\{(x,y) | \frac{x^2}{4} + y^2<1\}- \{(x,0) | x\leq 0\}$.
Then $R=2$, and $I_D=\{(0,0)\}$. The sets $D_r$, for $1<r<2$ are petals with vertex at the origin.

\begin{figure}[h] 
\centering
\resizebox{7cm}{!}{\includegraphics{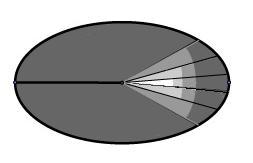}}    
\caption{The sets $D_r$ for the ellipse with a long axis removed.
The set $I_D$ is in the boundary}
\end{figure}

Note that $\bfr$ does not achieve the maximum is equivalent to $I_D\subset \bd D$. This does
not happen for convex $D$, as will be seen in the next section.

\section{Convex domains} \label{sec:convex}

{}From now on, we shall suppose that $D$ is a convex bounded domain.
This means that if $x,y\in D$, then
the segment $[x,y]$ is completely included in $D$.
There are several easy facts: $\overline D$ is a compact convex set,
the interior of $\overline D$ is $D$, and $\overline D$ is the convex hull of $\bd D$.

There is an alternative characterization for convex sets.  Let $v$ be a unit vector
in $\R^n$. Then the function $f(x)=\la x, v\ra$ achieves its maximum in
$\bd D$, say $c$. Then $f(x)\leq c$ for $x \in D$. Consider the half-space
  $$
  H_v^-= \{ x \in \R^n \, | \, f(x) < c\}\, .
  $$
Then $D\subset H_v^-$. We call
  $$
  H_v= \{ x \in \R^n \, | \, f(x) = c\}\,
  $$
a \emph{supporting hyperplane} for $D$. Note that $\bd D\cap H_v \not= \varnothing$.
Let also $H_v^+= \{ x \in \R^n \, | \, f(x) > c\}$.

\begin{lem} \label{lem:Int-Hv}
The convex set $D$ is the intersection
  $$
  \bigcap_{|v|=1} H_v^- \, ,
  $$
and conversely, any such intersection is a convex set.
Moreover,
 $$
 \overline D=\bigcap_{|v|=1} \overline{H}{}_v^- \, .
 $$
\end{lem}

\begin{proof}
The second assertion is clear, since the intersection of convex sets is convex.

For the first assertion, we have the trivial inclusion
$D\subset \bigcap_{|v|=1} H_v$. Now suppose $p\not\in D$. We have two cases:
 \begin{itemize}
 \item $p\not\in \overline D$. Then take $q\in \overline D$ such that $d(p,q)$ achieves its
 minimum, say $s>0$. Let $v$ be the unit vector from $p$ to $q$. Let $H_v$ be the hyperplane
 through $q$ determined by $v$. It is enough to see that the half-space $H_v^+$ is disjoint from
 $D$, since $p\in H_v^+$. Suppose that $x\in D\cap H_v^+$. Then the segment from $q$ to $x$
 should be entirely included in $\overline D$, but it intersects the interior of the ball
 of centre $p$ and radius $s$. This contradicts the choice of $q$.
 \item $p\in\bd D$. Consider $p_n\to p$, $p_n\not\in \overline D$. By the above, there
 are $q_n\in \bd D$ and vectors $v_n$ such that $D\subset H_{v_n}^- = \{ \la x-q_n,v_n\ra <
 0\}$. We take subsequences so that $q_n\to q\in\bd D$ and $v_n\to v$. So
 $D\subset \overline{H}{}_v^-=\{ \la x-q,v\ra  \leq 0\}$. But $D$ is open, so $D\subset H_v^-$.
 Moreover, as $d(p_n,D)\to 0$, then $d(p_n,q_n)\to 0$, so $p=q$, and the hyperplane
 determining $H_v$ goes through $p$, so $p\not\in H_v^-$ (actually $p\in H_v$).
 \end{itemize}
\end{proof}

\begin{rem} \label{rem:corner}
 The proof of Lemma \ref{lem:Int-Hv} shows that if $p\in \bd D$, then there is a supporting
 hyperplane $H_v$ through $p$. We call it a \emph{supporting hyperplane at $p$}, and we call
 $v$ a \emph{supporting vector at $p$}.
 When a point $p$ has several supporting hyperplanes, it is called a \emph{corner point}.
 The set
  $$
  \R_ + \cdot \{ v | v \text{ is supporting vector at } p\} \subset \R^n
  $$
 is convex.
 Note that if $\bd D$ is piecewise smooth, and $p\in \bd D$ is a smooth point, then
 $p$ is non-corner and the tangent space to $\bd D$ is the supporting hyperplane.
\end{rem}

Now we want to study the sets $D_r$ and $E_r$. First note that $\bfr$ is continuous on $D$. Therefore
 $$
 E_r\cap D =\{x | \bfr(x)\geq r \}
 $$
is closed on $D$.

\begin{prop} \label{lem:supremum}
If $D$ is convex, then $\bfr$ achieves its supremum $R$ at $D$.
Moreover, $I_D \cap D=E_R\cap D=\{p \, | \,  \bfr(p)=R\}$ and $I_D=E_R$ (which
is the closure of $E_R\cap D$).
\end{prop}

\begin{proof}
 Let $p\in I_D \cap \bd D$, and take a supporting
 hyperplane $H_v$ at $p$. We claim that the open semiball $B_R(p) \cap H_v^-\subset D$.
 If not, then there is a point $q\in \bd D$, $d(p,q)<R$, $q\in H_v^-$. Then all the
 segments $[q,x]$, with $x\in B_\epsilon(p)\cap \bd D$, have length $\leq r_0<$ (for
 suitable small $\epsilon$). Therefore, there is neighbourhood $U$ of $p$ such that
 $\bfr(x)\leq r_0$, $\forall x\in U\cap D$. Contradiction.

 Now all points in the ray $p+ t v$, $t\in (0,\epsilon)$ are not $r$-accessible for any
 $r<R$. Therefore they belong to $\bfr^{-1}(R)$. So $p$ is in the closure of $\bfr^{-1}(R)$,
 which is $E_R$.

 Therefore $I_D\cap \bd D \subset E_R$. Also, by continuity of $\bfr$ on $D$, we
 have that $E_r\cap D=\bfr^{-1}[r,\infty)$. Thus
 $I_D\cap D = \bigcap_{r<R} \bfr^{-1}[r,\infty) = \bfr^{-1}(R) =E_R \cap D$.
 All together, $I_D\subset E_R$. Obviously, $E_R\cap D=\bfr^{-1}(R)\subset I_D$,
 and taking closures, $E_R=\overline{E_R\cap D}\subset I_D$. So $I_D=E_R$.
 Finally, as $I_D$ is always non-empty, we have that $R$ is achieved by $\bfr$.
\end{proof}

Now we prove a useful result.
Given two points $P,Q$, we denote $\overrightarrow{PQ}=Q-P$ the vector from $P$ to $Q$.

\begin{lem} \label{prop:important}
 Let $p\in D$ and $r=\bfr(p)$. Let $[P,Q]$ be a segment of length $r$ with $P,Q\in \bd D$
 and $p\in [P,Q]$. Let $v_P,v_Q$ be supporting vectors at $P,Q$ respectively. Then
 \begin{enumerate}
  \item If $v_P,v_Q$ are parallel, then: $v_P=-v_Q$, $P,Q$ are non-corner points,
  $\overrightarrow{PQ} \parallel v_P$, and $r=R$.
  \item If $v_P,v_Q$ are not parallel, then: $\overrightarrow{PQ}$ is in the plane $\pi$ spanned by them,
  there is a unit vector $v \perp \overrightarrow{PQ}$, $v \in \pi$,
  such that for $H_v=\{\la x-p,v\ra =0\}$,
  it is $E_r \subset \overline{H}{}_v^-$; and $\bfr(x)<r$ for $x\in [P,Q]-\{p\}$ close to $p$.
  \end{enumerate}
\end{lem}

\begin{proof}
(1) Suppose first that $v_P,v_Q$ are parallel. So $D$ is inside the region between the parallel
 hyperplanes $H_{v_P}$ and $H_{v_Q}$. Clearly $v_P=-v_Q$. Let $x\in D$, and draw the segment
 parallel to $[P,Q]$ through $x$ with endpoints in the hyperplanes. It has length $r$. The
 intersection of this segment with $D$ is of length $\leq r$. Therefore $\bfr(x)\leq r$, for
all $x\in D$, so $R=r$.

If $\overrightarrow{PQ}$ is not parallel to
$v_P$, take a small vector $w$ such that $\la w,v_P\ra=0$, $\la w,\overrightarrow{PQ}\ra >0$. Let $t\in (0,1)$
so that $p=(1-t)P+ tQ$. Then $P'=P+tw \in H_{v_P}$ and $Q'=Q-(1-t)w \in H_{v_Q}$, and $p \in [P',Q']$.
First, $||\overrightarrow{P'Q'}||=|| \overrightarrow{PQ}-w||<||\overrightarrow{PQ}||=r$.
Also $P',Q' \not\in D$, so the segment $[P',Q']\cap D$ is of length
at most $||\overrightarrow{P'Q'}||$. Therefore
$\bfr(p)<r$, a contradiction.

The assertion that $P,Q$ are non-corner points is proved below.

(2) Suppose now that $v_P,v_Q$ are not parallel. Again $D$ is inside the region between the
hyperplanes $H_{v_P}$ and $H_{v_Q}$. Let $\pi$ be the plane spanned by $v_P,v_Q$. Let
$w$ be the projection of $\overrightarrow{PQ}$ on the orthogonal complement to $\pi$, and suppose $w\neq 0$.
Clearly $\la w,\overrightarrow{PQ}\ra >0$. Let $t\in (0,1)$
so that $p=(1-t)P+ tQ$. Then $P'=P+tw \in H_{v_P}$ and $Q'=Q-(1-t)w \in H_{v_Q}$, and $p \in [P',Q']$. So
$l([P',Q']\cap D) \leq ||\overrightarrow{P'Q'}|| <||\overrightarrow{PQ}||=r$, which is a contradiction.
Therefore $\overrightarrow{PQ}\in\pi$.

Let $v\in \pi$ be a unit vector such that $v \perp \overrightarrow{PQ}$.
Now consider unit vectors $e_1,e_2$ in $\pi$ so that $e_1\perp v_P$, $e_2\perp v_Q$,
The vector
 $$
 u= \frac{1}{\la e_1,v\ra  \, \la e_2,v\ra} (\la e_1,v\ra e_2 - \la e_2,v\ra e_1)
 $$
is perpendicular to $v$, hence parallel to $\overrightarrow{PQ}$. We arrange that $\la u,\overrightarrow{PQ}\ra <0$
by changing the sign of $v$ if necessary. Denote $H_v=\{\la x-p,v\ra =0\}$. Let us see that
this satisfies the statement.
Consider $w$ so that $\la w,v \ra > 0$. Let $w_1 = \frac{\la w,v\ra}{\la e_1,v\ra} e_1\in H_{v_P}$ and
$w_2 = \frac{\la w,v\ra}{\la e_2,v\ra} e_2\in H_{v_Q}$. Then
 $$
 w_2-w_1=\frac{\la w,v\ra}{\la e_1,v\ra  \, \la e_2,v\ra} (\la e_1,v\ra e_2 - \la e_2,v\ra e_1) = \la w,v\ra u\, ,
 $$
so $\la w_2-w_1, \overrightarrow{PQ}\ra =\la w,v\ra \la u, \overrightarrow{PQ}\ra <0$.
Set $P'=P+w_1$, $Q'=Q+w_2$. So $[P',Q']$ is parallel to $[P,Q]$, it goes through $p+w$,
and it is shorter than $[P,Q]$. 
So $H_v^+ \cap E_r=\varnothing$.

For the last assertion, we write
$\overrightarrow{PQ} = a_1 e_1 + a_2 e_2$, where $a_1,a_2\neq 0$.
Let $P'=P+ x e_1\in H_{v_P}$, $Q'= Q+ ye_2\in H_{v_Q}$.
The condition $p\in [P',Q']$ is equivalent to $p,P',Q'$ being aligned, which is rewritten as
 \begin{equation} \label{eqn:1}
 xy +(1-t)a_2 x-ta_1 y =0 \, .
 \end{equation}
Now, the condition  $||\overrightarrow{P'Q'}||=|| \overrightarrow{PQ}+ye_2-xe_1||<||\overrightarrow{PQ}||=r$
for small $r$ is achieved if $\la \overrightarrow{PQ},ye_2-xe_1\ra <0$.
This is a linear equation of the form $\alpha_1 x +\alpha_2 y<0$. The intersection of such half-plane
with the hyperbola (\ref{eqn:1}) is non-empty except if $\alpha_1 x +\alpha_2 y=0$ is tangent to the hyperbola
at the origin. So $(\alpha_1,\alpha_2)$ is a multiple of $((1-t)a_2, -ta_1)$.
This determines $t$ uniquely. So for $s\neq t$ (and close to $t$),
we have that $p_s=(1-s)P+sQ$ satisfies $\bfr(p_s)<r$.
(Note incidentally, that it cannot be $\overrightarrow{PQ} \parallel v_P$. If so, then $\alpha_1=0$, and then
$(1-t)a_2=0$, so $t=1$, which is not possible.)

Now we finish the proof of (1). Suppose that $Q$ is a corner point. Then we can choose another supporting
vector $v_Q'$. On the one hand $\overrightarrow{PQ}\parallel v_P=-v_Q$. On the other, as $v_P\not\parallel v_Q'$,
we must have $\overrightarrow{PQ} \not\parallel v_P$, by the discussion above. Contradiction.
\end{proof}

\begin{teo} \label{teo:convex}
 The sets $D_r$, $E_r$ are convex sets, for $r\in [0,R]$, $R=\max \bfr$.
 Moreover, $\bd D_r \cap D$ is $\bfr^{-1}(r)$, for $r\in (0,R)$.
\end{teo}

\begin{proof}
The assertion for $E_r$ follows from that of $D_r$: knowing that $D_r$ is convex, then
 $$
 E_r= \overline{\bigcap_{\epsilon>0} D_{r-\epsilon}}
 $$
 is convex since the intersection of convex sets is convex, and the closure of a convex set is convex.

Let $0<r<R$, and let us  see that $D_r$ is convex. Let $p\not\in D_r$. Then $\bfr(p)\leq r$. By
Lemma \ref{prop:important}, there is a segment $[P,Q]$ of length $r$, with $P,Q\in \bd D$, $v_P \not\parallel
v_Q$, and
a vector $v\perp \overrightarrow{PQ}$ such that $E_r\subset \overline{H}{}^-_v$. Then $D_r\subset H_v^-$, and
$p\not\in H_v^-$. So $D_r$ is the intersection of half-spaces, hence convex.

For the last assertion, note that the continuity of $\bfr$ implies that $D\cap \bd D_r \subset \bfr^{-1}(r)$.
For the reversed inclusion, suppose that $\bfr(p)=r$, but $p\not\in \bd D_r$. Then there is some $\epsilon>0$
so that $B_\epsilon(p) \subset \bfr^{-1}(0,r]$. Now $\bfr^{-1}[r,\infty)$ is convex, so it is the closure
of its interior, call it $V$. Therefore $V\cap B_\epsilon(p)$ is open, convex, and contains $p$ in its
adherence. Moreover $V\cap B_\epsilon(p)\subset \bfr^{-1}(r)$.
But this is impossible, since an easy consequence of Lemma \ref{prop:important}
is that $\bfr^{-1}(r)$ has no interior for any $r\in (0,R)$.
\end{proof}

\begin{prop} \label{prop:envelope}
Suppose $D$ is a convex \emph{planar} set. Let $r\in (0,R)$. Then $\bd D_r$ is
the envelope of the segments of length $r$ with endpoints at $\bd D$.
\end{prop}

\begin{proof}
As we proved before, the boundary of $D_r$ is $\bfr^{-1}(r)$, so the points of $\bd D_r$ are $r$-accessible,
but not $r'$-accessible for $r'<r$. 
Let $p\in \bd D_r$ be a smooth point. Then there is a segment of length $r$ and $D_r$ is at one side of it.
Therefore the segment is tangent to $\bd D_r$ at $p$.
\end{proof}

\section{Strictly convex domains} \label{sec:strictly-convex}

Recall that $D$ is \emph{strictly convex} if there is no segment included in its boundary.
We assume that $D$ is strictly convex in this section.
Therefore, for each unit vector $v$, there is a unique point of contact
$H_v\cap\bd D$. We define the function
  $$
  g: S^{n-1}\to \bd D \, .
  $$

\begin{lem}
 If $D$ is strictly convex, then $g$ is continuous.
\end{lem}

\begin{proof}
Let $v_n\in S^{n-1}$, $v_n\to v$. Consider $p_n=g(v_n)\in \bd D$, and the supporting hyperplane
$\la x-p_n,v_n\ra \leq 0$. Let $p=g(v)$, with supporting hyperplane $\la x-p,v\ra \leq 0$.
After taking a subsequence, we can suppose $p_n\to q\in \bd D$.
Now $p\in \overline D \implies \la p-p_n,v_n\ra \leq 0$, and taking limits, $\la p-q,v\ra \leq 0$.
On the other hand, $p_n\in \overline D \implies \la p_n-p,v\ra \leq 0$, and taking limits, $\la q-p,v\ra\leq 0$.
So $\la q-p,v\ra=0$. By strict convexity, $q=p$, so $g(v_n) \to g(v)$, and $g$ is continuous.
\end{proof}

Now suppose that $\bd D$ is $C^1$. Then for each point $p\in \bd D$, there
is a normal vector $\nn(p)$. We have a well defined function
  $$
  \phi: \bd D\to S^{n-1}, \quad \phi(p)=\nn(p)\, .
  $$
Note that $p\in H_{\nn(p)}\cap \overline D$. Therefore if $D$ is $C^1$ and strictly convex,
both $\phi$ and $g$ are defined and inverse to each other.

In general, for $D$ convex, there are pseudo-functions $g:S^n\to \bd D$, $\phi:\bd D\to S^n$.
A pseudo-function assigns to each point $v\in S^n$ a subset $g(v)\subset \bd D$
in such a way that the graph $\{(v,p)\, | \, p\in g(v)\}$ is closed. The inverse of
a pseudo-function is well-defined, and $g$ and $\phi$ are inverse to each other.
The set $\phi(p)$ is the set of supporting vectors at $p$ (see Remark \ref{rem:corner}).

\begin{lem}\label{lem:1} Suppose $D$ strictly convex.
For all $0<r<R$, $\bd D_r \cap \bd D=\varnothing$, so $\bd D_r=\bfr^{-1}(r)$.
\end{lem}

\begin{proof}
Take a point $p\in \bd D$, and let $H_v$ be a supporting hyperplante.
Consider a small ball $B$ around $p$ of radius $\leq r/2$.
By strict convexity, $d(\bd B\cap D, H)=\epsilon_0>0$. Now we
claim that $B_{\epsilon_0}(p)\cap D$ does not intersect $D_r$, so $p\not\in \overline D_r$.
Let $q\in B_{\epsilon_0}(p)\cap D$, and consider a line $l$ parallel to $H$ through $q$. The
segment $l\cap B$ has endpoints $P,Q\in \bd B$. But $d(P,H)=d(Q,H)<\epsilon_0$, so
$P,Q\not\in D$. So the connected component
$[P,Q]\cap D$ has length $<||\overrightarrow{PQ}||<r$, and $q$ is $r'$-accessible
for some $r'<r$.
\end{proof}

\begin{cor}
For $D$ strictly convex,
$\bfr: \overline D \to \R_{\geq 0}$ is continuous.
\end{cor}

\begin{proof}
 By Remark \ref{rem:a}, $\bfr$ is continuous on $D$.
 The continuity at $\bd D$ follows from the proof of Lemma \ref{lem:1}.
\end{proof}

Therefore, if $D$ is strictly convex, then
 $$
 I_D=E_R=\bfr^{-1}(R)\, .
 $$
As $I_D\subset D$, we have that $I_D$ does not touch $\bd D$.

\begin{teo} \label{lem:2}
Let $D$ be strictly convex. For all $0<r<R$, $D_r$ is strictly convex.
\end{teo}

\begin{proof}
Suppose that $\bd D_r$ contains a segment $l$. Let $p$ be a point in the interior of $l$.
As it is $r$-accessible, there is a segment $[P,Q]$ of length $r$ through $p$, where $P,Q\in\bd D$.
By Lemma \ref{prop:important}, $v_P,v_Q$ are not parallel, and all points in $[P,Q]$ different from
$p$ are $r'$-accessible for some $r'<r$. Therefore $l$ is transversal to $[P,Q]$. Let $H_v$ be the
hyperplane produced by Lemma \ref{prop:important} (2). Then all points at one side of $H_v$ are
$r'$-accessible for some $r'<r$, hence $l$ cannot be transversal to $H_v$, so $l\subset H_v$.

Now let $x\in l$, $x\neq p$. Consider the segment parallel to $[P,Q]$ through $x$, call it $\sigma$.
It has length $r$
and endpoints at $H_{v_P}, H_{v_Q}$. But $D$ is strictly convex, so it only touches the supporting
hyperplanes at one point. Hence $\sigma \cap D$ is strictly contained in $\sigma$. Therefore
$\bfr(x)<r$. Contradiction.
\end{proof}

\section{Set of maximum inaccessibility} \label{sec:D-R}

In this section we suppose that $D$ is convex. Then $\bfr$ is continuous on
$D$ and it achieves its maximum $R$ on $D$. Then $I_D =E_R$ and
$I_D\cap D=E_R\cap D=\bfr^{-1}(R)$, by Proposition \ref{lem:supremum}.

We want to characterize the case where $I_D$ contains interior.
Let us see an example where this situation happens. Let $D$ be a rectangle.
In this case $R$ is the length of the shortest edge of the rectangle, and we have
an open set with $\bfr (p)=R$ (see Figure \ref{fig:rectangle}). Note that
it might happen that $\bd E_R$ intersects $\bd D$.

\begin{prop}\label{prop:interior}
  If $I_D$ has non-empty interior, then $\bd D$ contains two open subsets
  which are included in parallel hyperplanes, which are at distance $R$.
\end{prop}

\begin{proof}
Consider an interior point $p\in I_D$, so $\bfr(p)=R$. Take
a segment $l=[P,Q]$ of length $R$ with endpoints $P,Q\in \bd D$.
Let $v_P,v_Q$ be vectors orthogonal to the supporting hyperplanes at $P,Q$. By Lemma
\ref{prop:important}, if they are not parallel, then there is a hyperplane through $p$ such that
$E_R$ is contained in one (closed) half-space. This is not possible, as $p$ is an interior point
of $E_R$. So $v_P,v_Q$ are parallel, and $\overrightarrow{PQ} \parallel v_P$.
Now take any point $x$ close to $p$, and consider the segment
$[P',Q']$ through $x$ parallel to $[P,Q]$, which has endpoints in $H_{v_P},H_{v_Q}$.
If $[P',Q']\cap D$ is properly contained in $[P',Q']$, then $\bfr(x)<R$, which contradicts
that $x\in E_R$. So $P'\in H_{v_P}$, $Q'\in H_{v_Q}$, and
$\bd D$ contains two open subsets in $H_{v_P}, H_{v_Q}$ around $P,Q$, respectively.
\end{proof}

\begin{teo} \label{teo:main-DR}
Let $D$ be a strictly convex bounded domain, $R=\max \bfr$. Then
 $I_D$ is a point or a segment. 
\end{teo}

\begin{proof}
 Suppose that $I_D$ is not a point. As it is convex by Theorem \ref{teo:convex},
 it contains a maximal segment $\sigma$. Let us see that it cannot contain two different
 (intersecting) segments. Let $p\in \sigma$ be an interior point of the segment.
 By Lemma \ref{prop:important}, if we draw the segment $[P,Q]$ of length
 $R$ through $p$, we have the following possibilities:
 \begin{itemize}
  \item $v_P,v_Q$ are parallel. Then $\overrightarrow{PQ} \parallel v_P$. Then any point
  $x\not\in [P,Q]$ lies in a segment $[P',Q']$ parallel to $[P,Q]$, with $P'\in H_{v_P}$
   and $Q'\in H_{v_Q}$. By strict convexity, $l([P',Q']\cap D)<R$, so $\bfr(x)<R$. That is, $E_R\subset [P,Q]$.
  \item $v_P,v_Q$ are non-parallel. Then there is a hyperplane $H_v$ through $p$ such that
  $E_R\subset \overline{H}{}_v^-$. As $p$ is an interior point of $\sigma$, $\sigma$ does not cross $H_v$,
  so $\sigma \subset H_v$. Now let $x\in \sigma$, and consider the segment $[P',Q']$ parallel to $[P,Q]$
  through $x$, with length $R$, $P'\in H_{v_P}$
   and $Q'\in H_{v_Q}$. If $x\not\in [P,Q]$ then
   strict convexity gives $l([P',Q']\cap D)<R$, so $\bfr(x)<R$.
   That is, $E_R\subset [P,Q]$. Note that Lemma \ref{prop:important} (2) gives in this case that $E_R=\{p\}$.
 \end{itemize}
\end{proof}

Let us see an example where $E_R$ is a segment.
Let $D$ be the ellipse with equation $\frac{x^2}{a^2} + \frac{y^2}{b^2} <1$, where $2a>2b$.
Then $R=2b$
and $E_R$ is a segment contained in the short axis, delimited by the intersection of the axis with
the perpendicular segments of length $R$ with endpoints in the ellipse.

\begin{figure}[h] 
\centering
\resizebox{6cm}{!}{\includegraphics{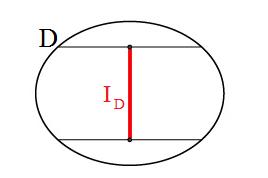}}    
\caption{Ellipse}
\end{figure}

All points $(x,y)\in D$
 with $x\neq 0$ can be reached by vertical segments of length $<R=2b$. Now let
 $x_0= b$, $y_0=b\sqrt{a^2-b^2}/a$. If $y\in (-b,-y_0)\cup (y_0,b)$ then the point $(0,y)$ is
 $r$-accessible (with a horizontal segment) with $r<R$.
Now let $y\in [-y_0,y_0]$, and consider a line through $(0,y)$. Let us parametrize it as
 $$
 r(s)=(s\, a\, \cos \theta, y + s\, b\, \sin \theta)\, ,
 $$
with $\theta$ fixed.
The intersection with the ellipse are given by
$s=-\frac{y}{b} \sin\theta\pm\sqrt{1- \frac{y^2}{b^2}\cos^2 \theta}$.
So the square of the distance between the two points is
 \begin{eqnarray*}
 l(\theta)^2 &=&4(1- \frac{y^2}{b^2}\cos^2\theta) (a^2\cos^2\theta+b^2\sin^2\theta) \\
  &=& 4(1-\frac{y^2}{b^2} T) ((a^2-b^2)T +b^2) \, ,
 \end{eqnarray*}
where $T= \cos^2\theta \in [0,1]$. The minimum of this degree $2$ expression on $T$
happens for a negative value of $T$, therefore, we only need to check the values $T=0,1$.
For $T=0$, we get $4b^2$; for $T=1$, we get $4(1- \frac{y^2}{b^2})a^2
\geq 4(1- \frac{y_0^2}{b^2})a^2=4b^2$. So $l(\theta)^2\geq 4b^2$.

\medskip

A consequence of Theorem \ref{teo:main-DR}
is the following: for a strictly convex bounded domain $D$, if
  $I_D$ is not a point then there are two non-corner points $P,Q\in \bd D$ with
  parallel tangent hyperplanes which are moreover perpendicular to $\overrightarrow{PQ}$.

\begin{cor} \label{cor:main-DR-cor}
Suppose $D$ is a \emph{planar} convex bounded domain (not necessarily strictly convex). If
  $I_D$ is not a point then there are two non-corner points $P,Q\in \bd D$ with
  parallel tangent hyperplanes which are moreover perpendicular to $\overrightarrow{PQ}$.
\end{cor}

\begin{proof}
Following the proof of Theorem \ref{teo:main-DR}, we only have to rule out case (2).
As the hyperplane $H_v$ is now of dimension $1$, we have $\sigma\subset [P,Q]=H_v\cap \overline{D}$.
But Lemma \ref{prop:important} says also that $E_R\cap [P,Q]=\{p\}$. So $E_R$ does not contain
a segment, i.e. it is a point.
\end{proof}

So, for a convex polygon $D$, if it does not have parallel sides, then $I_D$ is a point.

Corollary \ref{cor:main-DR-cor} is not true in dimension $\geq 3$. Take a triangle $T\subset \R^2$
and consider $D=T\times [0,L]$ for large $L$. For $T$, denote $I_T=\{p\}$. Then
$D$ has $I_D=\{p\}\times [a,b]$, for some $0<a<b<L$. Certainly, there are two parallel faces (base and
top), but we slightly move one of them to make them non-parallel, and $I_D$ is still a segment.

\begin{figure}[h] 
\resizebox{8cm}{!}{\includegraphics{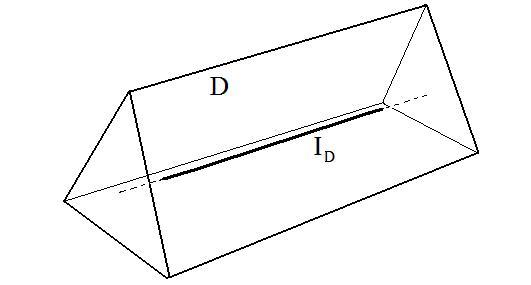}}    
\caption{$I_D$ can be positive dimensional}
\end{figure}

One can make this construction to have $I_D$ of higher dimension (not just a segment),
e.g. by considering $T\times [0,L]^N$, $N>1$.

\section{Polygons} \label{sec:triangles}

In this section we want to study in detail the case of convex polygons in the
plane, and to give some answers in the case of triangles.
The starting point is the case of a sector.

\begin{lem} \label{lem:sector}
  Fix $\lambda\in \R$.
  Let $D$ be the domain with boundary the half-lines $(x,0)$, $x\geq 0$ and
  $(\lambda y, y)$, $y\geq 0$. Let $r>0$.
  Then the boundary of $D_r$ is the curve:
 \begin{equation}\label{eqn:curve}
 \left\{ \begin{array}{l}
 x=r(\cos^3\theta + \lambda (\sin^3 \theta+ 2 \sin\theta\cos^2\theta)) \\
 y=r(\sin^3\theta - \lambda\sin^2\theta\cos\theta)
 \end{array}\right.
 \end{equation}
\end{lem}

\begin{proof}
$D$ is not a bounded domain, but the theory works as well in this case.
To find the boundary of $D_r$, we need to take the envelope of the segments of
length $r$ with endpoints laying on the half-rays, according to Proposition
\ref{prop:envelope}. Two points at $(a,0)$ and $(\lambda b,b)$ are at distance $r$
if
 $$
 (\lambda b-a)^2+b^2=r^2 \,.
 $$
So $\lambda b-a=-r\cos\theta$,  $b=r\sin\theta$, i.e. $a=\lambda r \sin\theta + r \cos \theta$.
The line which passes through $(\lambda b, b)$ and $(a,0)$ is
 $$
  r\sin\theta \ x + r\cos\theta \ y = r^2 \sin\theta\cos\theta + r^2\lambda\sin^2\theta\, .
 $$
We are going to calculate the envelope of these lines. Take the derivative and solve the system:
$$
 \left\{ \begin{array}{l}
 r\sin\theta \ x+r\cos\theta \ y = r^2\sin\theta\cos\theta + r^2\lambda\sin^2\theta \\
 r\cos\theta \ x-r\sin\theta \ y = -r^2\sin^2\theta+r^2\cos^2\theta + 2r^2\lambda\sin\theta\cos\theta\, .
 \end{array}\right.
 $$
We easily get the expression in the statement.
The region $D_r$ is the unbounded region with boundary the curve (\ref{eqn:curve})
and the two half-rays.
\end{proof}

We call the curve in Lemma \ref{lem:sector} a {\em $\lambda$-bow} (or just a bow). Let $\lambda=\cot \alpha$,
$\alpha \in (0,\pi)$. If $\lambda <0$, we are dealing with an obtuse angle, and $\theta \in [0,\pi-\alpha]$.
If $\lambda =0$, we have a right angle, and $\theta \in [0,
\frac{\pi}{2}]$. Finally, an acute angle happens for $\lambda >0$. In this case, $\theta \in [\frac{\pi}{2} -\alpha,
\frac{\pi}{2}]$. (Note that $\theta$ is the angle between the segment and the negative horizontal axis,
in the proof of Lemma \ref{lem:sector}.)

\begin{figure}[h] 
\resizebox{5cm}{!}{\includegraphics{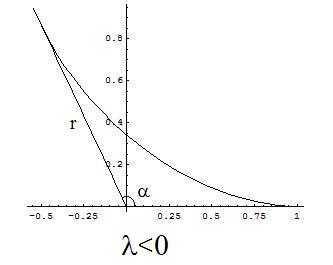}}    
\resizebox{5cm}{!}{\includegraphics{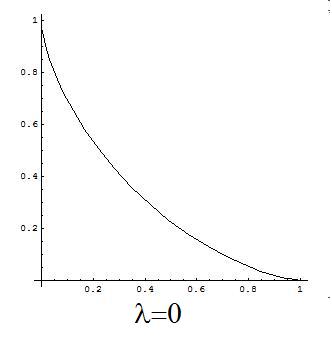}}    
\resizebox{5cm}{!}{\includegraphics{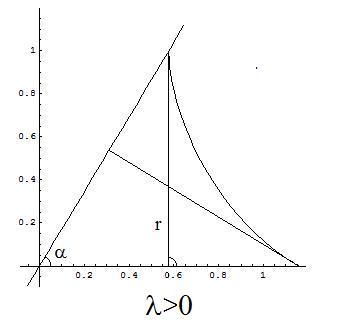}}    
\caption{$\lambda$-bows with $r=1$}
\end{figure}

As an application, we prove the following:

\begin{cor}
  Let $D$ be a planar convex polygon. Then the sets $D_r$, $0<r<R$, and $I_D$ if it is not a point,
  have boundaries
  which are piecewise $C^1$, and whose pieces are $\lambda$-bows 
  and (possibly) segments in the sides of $\bd D$.
  In particular, these domains are strictly convex when $\bd D_r$ does not intersect $\bd D$.
\end{cor}

\begin{proof}
 Let $l_1,\ldots, l_k$ be the lines determined by prolonging
 the sides of the polygon. Consider $l_i$, $l_j$. If they intersect,
 consider the sector that they determine in which $D$ is contained. Lemma \ref{lem:sector}
 provides us with a (convex) region $D_r^{ij}$. If $l_i,l_j$ are parallel and $r< d(l_i,l_j)$
 then set $D_r^{ij}=D$, and if $l_i,l_j$ are parallel and $r  \geq d(l_i,l_j)$
 then set $D_r^{ij}=\varnothing$. It is fairly clear that
  $$
   D_r= \bigcap_{i\neq j} D_r^{ij} \, .
  $$

 To see the last assertion, note that at any smooth point $p\in \bd D_r$, we have strict convexity
 because of the shape of the bows given in Lemma \ref{lem:sector}. If $p\in \bd D_r$ is a non-smooth
 point, then it is in the intersection of two such curves. This means that there are segments $\sigma_1,\sigma_2$
 of length $r$ where $\sigma_1$ has endpoints at lines $l_{i_1},l_{j_1}$ and
 $\sigma_2$ has endpoints at lines $l_{i_2},l_{j_2}$. Moreover, the endpoints should be actually in
 the sides of $D$ (otherwise $p$ would be $r'$-accessible for some $r'<r$). In particular, this means that
 $\sigma_1$, $\sigma_2$ cannot be parallel. As such segments are tangent to the bows, the curves
 intersect transversely at $p$, and $p$ is a corner point.

 A similar statement holds for $I_D=E_R$, when it is not a point, by doing the above reasoning for $r=R$.
\end{proof}

In particular, we see that $I_D$ cannot be a segment for polygons.

For instance, when $D$ is a rectangle of sides $a\geq b$, then $R=b$. We draw the bows at the vertices,
to draw the set $I_D=E_R$.

\begin{figure}[h] 
\resizebox{6cm}{!}{\includegraphics{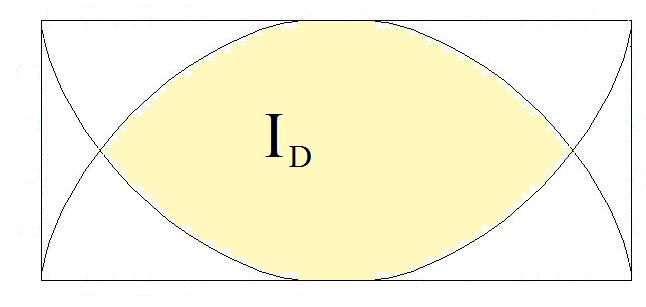}}    
\caption{For a rectangle, $I_D$ has interior} \label{fig:rectangle}
\end{figure}

Note that $I_D$ intersects $\bd D$ if and only if $a\geq 2b$.

\bigskip

It would be nice to have a function
 $$
 I_D=(I_1,I_2)=I_D((x_1,y_1),(x_2,y_2),\ldots, (x_k,y_k)) \in \R^2\, ,
 $$
which assigns the value of $I_D$ given the vertices $(x_i,y_i)$ of a $k$-polygon. Such
function is only defined for polygons with non-parallel sides. 

We shall produce the formula for $I_D$ for the case of an \emph{isosceles triangle}.
Consider an isosceles triangle of height $1$, and base $2\lambda>0$. Put the vertices at
the points $A=(0,0)$, $B=(2\lambda,0)$ and $C=(\lambda,1)$.
By symmetry, the point $I_D$ must lie in the vertical axis $x=\lambda$. Moreover, the
segment of length $R$ through $I_D$ tangent to the bow corresponding to $C$ must be
horizontal.
This means that $I_D=(\lambda, I_2)$ where $\frac{R}2=\lambda (1-I_2)$.
So
 $$
 I_D= (I_1,I_2)=(\lambda, I_2(\lambda))=\left( \lambda, 1- \frac{R}{2\lambda}\right) \, .
 $$
The sector corresponding to $A$ is that of Lemma \ref{lem:sector}, and
the point $I_D$ should lie in its $\lambda$-bow, which is
the curve given in Lemma \ref{lem:sector} for the value $r=R$.  Hence
\begin{eqnarray*}
  \lambda &=& R (\cos^3\theta + \lambda (\sin^3 \theta+ 2 \sin\theta\cos^2\theta)) , \\
  1- \frac{R}{2\lambda} &=& R (\sin^3\theta - \lambda\sin^2\theta\cos\theta)\, .
\end{eqnarray*}
Eliminating $R$, we get
 $$
 \lambda^2 \sin^2\theta \cos \theta + 2\lambda \sin\theta \cos^2\theta + \cos^3\theta -\frac12 =0
 $$
i.e.
 \begin{equation}\label{eqn:la}
 \lambda=\frac{-2\cos^2\theta + \sqrt{2\cos\theta}}{2\sin\theta \cos\theta}
 \end{equation}
(the sign should be plus, since $\lambda>0$). Note that for an equilateral triangle, $\lambda=\frac{1}{\sqrt{3}}$,
$I_2=\frac13$, $\theta=\frac\pi3$ and $R=\frac{4}{3\sqrt{3}}$.

Also
  \begin{equation}\label{eqn:R}
 R=\frac{\lambda}{ \cos^3\theta + \lambda (\sin^3 \theta+ 2 \sin\theta\cos^2\theta)} \, .
  \end{equation}
 One can check the following formula:
 \begin{equation}\label{eqn:I2}
  I_2= 1- \frac{R}{2\lambda}= \lambda \frac{\sin^3\theta - \lambda\sin^2\theta \cos \theta}{\cos^3\theta +
  \lambda (\sin^3\theta + 2 \sin \theta \cos^2\theta)} \, .
   \end{equation}
This locates the point $I_D=(\lambda(\theta), I_2(\lambda(\theta)))$.

\begin{rem}
 Do the change of variables $\cos\theta=\frac{1-u^2}{1+u^2}$, $\sin \theta=\frac{2u}{1+u^2}$, to get
 algebraic expressions for $I_D$. It is to be expected that this algebraicity property holds
 for a general triangle.
\end{rem}

Recall the position of the ortocentre, incentre, baricentre and circumcentre
 \begin{eqnarray*}
 H &=& \left(\lambda, \lambda^2 \right). \\
 I&=&\left(\lambda,  \frac{\lambda}{\lambda+\sqrt{\lambda^2+1}}\right). \\
 G&=&\left(\lambda, \frac13\right), \\
 O&=&\left(\lambda, \frac{1-\lambda^2}{2}\right).
 \end{eqnarray*}
We draw the height of the point $H,I,G,O,I_D$ as a function of $\lambda$ : 

\begin{figure}[h] 
\resizebox{12cm}{!}{\includegraphics{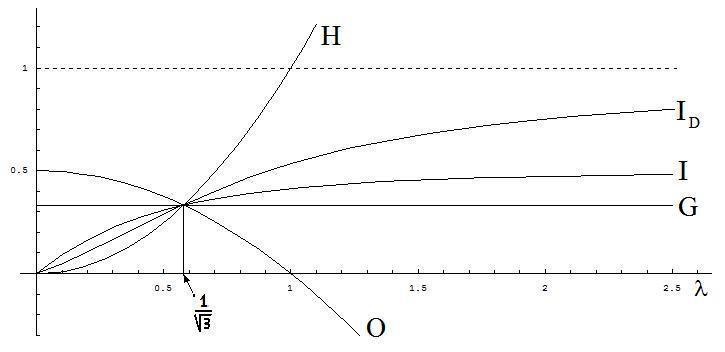}}    
\caption{Notable points of a triangle of height $1$ and base $2\lambda$}
\end{figure}

A simple consequence is that these $5$ points are distinct for an isosceles triangle
which is not equilateral. We conjecture that this is true for a non-isosceles triangle.

Note the asymptotic for an isosceles triangle. For $\lambda \sim 0$, we have that (\ref{eqn:la})
implies $\cos^3\theta \sim \frac12$. Now (\ref{eqn:R}) and (\ref{eqn:I2}) give that
$R\sim 2\lambda$ and
  $$
 I_2(\lambda) \sim \frac{\sin^3\theta}{\cos^3\theta} \lambda  \sim ( 2^{2/3}-1)^{3/2} \lambda \, .
  $$
Rescale the triangle to have base $b=2$ and height $h=\frac{1}{\lambda}$. Then when $h$ is large,
the point $I_D$ approaches to be at distance $( 2^{2/3}-1)^{3/2}=0.4502$ to the base, and $R \sim 2$.
Also, for $\lambda \to \infty$, we have $I_2(\lambda)\to 1$.

\begin{rem} \label{rem:not-continuous}
 Consider a rectangle $D$ with vertices $( \pm a, \pm 1)$, with $a\gg 1$.
 Then $I_D$ has interior (see Figure \ref{fig:rectangle}). Moving slightly the vertices at the left,
we get an isosceles trapezoid $Z_\epsilon$, with vertices $( -a,\pm (1-\epsilon)), (a,\pm 1)$, for
$\epsilon>0$.
Consider the triangle $T_\epsilon$ obtained by prolonging the long sides of $Z_\epsilon$, i.e.
with vertices $(a-2a/\epsilon,0), (a,\pm 1)$. By the above, the point $I_{T_\epsilon}
\sim (a-0.4502, 0)$. As $R\sim 2$, we have that $I_{Z_\epsilon}=I_{T_\epsilon}$.

By symmetry, if we consider the isosceles trapezoid $Z_\epsilon'$ with vertices
$( -a,\pm 1 ), (a,\pm (1-\epsilon))$, then $I_{Z_\epsilon'} \sim (-a+0.4502,0)$.

The polygons $Z_\epsilon$ and $Z_\epsilon'$ are nearby, but their points of maximum inaccessibility are
quite far apart. So the map $D\mapsto I_D$ cannot be extended continuously (in any reasonable topology)
to all polygons with $4$ sides.
\end{rem}

%
%

\end{document}